\documentclass[11pt,reqno]{amsart}
\usepackage{amssymb,latexsym}
\usepackage{graphicx}
\usepackage{url}

\calclayout

\makeatletter
\newcommand{\monthyear}[1]{%
  \def\@monthyear{\uppercase{#1}}}
\newcommand{\volnumber}[1]{%
  \def\@volnumber{\uppercase{#1}}}
\AtBeginDocument{%
\def\ps@plain{\ps@empty
  \def\@oddfoot{\@monthyear \hfil \thepage}%
  \def\@evenfoot{\thepage \hfil \@volnumber}}
\def\ps@firstpage{\ps@plain}
\def\ps@headings{\ps@empty
  \def\@evenhead{%
    \setTrue{runhead}%
    \def\thanks{\protect\thanks@warning}%
    \uppercase{Preprint}\hfil}%
  \def\@oddhead{%
    \setTrue{runhead}%
    \def\thanks{\protect\thanks@warning}%
    \hfill\uppercase{Clusters of Integers with Equal Total Stopping Times}}%
  \let\@mkboth\markboth
  \def\@evenfoot{%
    \thepage \hfil \@volnumber}%
  \def\@oddfoot{%
    \@monthyear \hfil \thepage}%
  }%
\footskip=25pt
\pagestyle{headings}%
}
\makeatother

\newcommand{\Z}{{\mathbb Z}}

\theoremstyle{plain}
\numberwithin{equation}{section}
\newtheorem{thm}{Theorem}[section]
\newtheorem{theorem}[thm]{Theorem}
\newtheorem{lemma}[thm]{Lemma}
\newtheorem{example}[thm]{Example}

\newtheorem{proposition}[thm]{Proposition}

\newtheorem{corollary}[thm]{Corollary}
\newtheorem{remark}[thm]{Remark}

\begin{document}
\monthyear{November 2017}
\volnumber{55, 4}
\setcounter{page}{1}

\title{Clusters of Integers with Equal Total Stopping Times in the 3x + 1 Problem}
\author{Mark D. LaDue}
\address{Mathematics and Computer Science Faculty\\
         The Magnum Opus School\\
         Southlake, TX\\
         76092, USA}
\email{mladue@gatech.edu}

\begin{abstract}
The clustering of integers with equal total stopping times has long been observed in the
3x + 1 Problem, and a number of elementary results about it have been used repeatedly in
the literature \cite{pA2000, jL1985, jL2010}.  In this paper we introduce a simple
recursively defined function $C: \Z^+ \rightarrow \{0, 1\}$, and we use it to give a
necessary and sufficient condition for pairs of consecutive even and odd integers
to have trajectories which coincide after a specific pair-dependent number of steps.
Then we derive a number of standard total stopping time equalities,
including the ones in \cite{lG1985}, as well as several novel results.
\end{abstract}

\maketitle

\section{Introduction}\label{S:intro}
The \emph{3X + 1 Problem} may be stated concisely in terms of the function $T: \Z^+ \rightarrow \Z^+$.
We define
\[
   T(m) = 
   \begin{cases}
   m / 2 &\text{if $m$ is even;}\\
   (3m + 1) / 2 &\text{if $m$ is odd.}
   \end{cases}
\]
For each positive integer $k$ we define $T^k(m)$ to be the result of
composing $T$ with itself $k$ times and evaluating the resulting function at $m$.  By
convention we also define $T^0(m) = m$.

We further define the \emph{trajectory}
of $m \in \Z^+$ to be the sequence of values $(T^k(m))_{k \ge 0}$ and the \emph{parity sequence}
of $m \in \Z^+$ to be the sequence of values $(v_k(m))_{k \ge 0}$, where
\[
   v_k(m) = 
   \begin{cases}
   0 &\text{if $T^k(m)$ is even;}\\
   1 &\text{if $T^k(m)$ is odd.}
   \end{cases}
\]
A \emph{parity vector} of length $K$ for $m \in \Z^+$ is then defined to be the first $K$
entries of the parity sequence for $m$.

Let $F = \{m \in \Z^+ \mid T^k(m) = 1\text{ for some }k \ge 0\}$.  Since $T^n(2^n) = 1$, 
$2^n \in F$ for all $n \ge 0$.  If $m \in F$, then $T^k(m) = 1$ for some $k \ge 0$, and we
define the \emph{total stopping time} of $m$ to be 
\[
   \sigma_{\infty}(m) = \text{smallest element of }\{k \ge 0 \mid T^k(m) = 1\}.
\]
If $m \notin F$, we define $\sigma_{\infty}(m) = \infty$.  The \emph{Collatz Conjecture} asserts
that $F = \Z^+$, or equivalently, $\sigma_{\infty}(m)$ is finite for all $m \in \Z^+$, and
the \emph{3X + 1 Problem} is to determine the truth of this conjecture.

If a set of integers $\{m_{1}, m_{2}, \dots, m_{k}\}$ satisfies $\sigma_{\infty}(m_{i}) =
\sigma_{\infty}(m_{j})$ for $1 \le i, j \le k$, we say that these integers form a \emph{cluster}.
We shall be interested mainly in clusters of two and three consecutive integers, but our
definition places no restrictions on how clusters are formed.

Before we discuss previous results, a word about terminology is in order.  The astute reader
may have noticed that the papers by Garner \cite{lG1985} and Gao \cite{gG1993} both refer to
\emph{heights} in the 3X~ +~ 1 Problem.  It turns out that Garner's definition of height
(\cite{lG1985}, p.~57) is the same as our definition of the total stopping time, while Gao's
definition of height (\cite{gG1993}, p.~262) agrees with the authoritative definition of
Lagarias (\cite{jL2012}, p.~2), which is different.  Thus in what follows, when we discuss
Garner's work, we are justified in referring to total stopping times instead of heights.
This will allow us to avoid any confusion which may result from the conflicting definitions.

While results about clusters have appeared in previous work, very few papers have considered
the topic explicitly.  In \cite{pA2000} Andaloro gave a different, but equivalent, formulation
of the 3X~ +~ 1 Problem.  In the course of proving his theorems about sufficient sets, he included
a number of results about clusters of integers with equal total stopping times.  Since his function\\
$T: \Z^+ \rightarrow \Z^+$ differs from the standard one that we use, his total stopping times differ
from ours.  Consequently, his clusters differ from, but are related to, ours.

In \cite{lG1985} Garner used parity vectors (first introduced and studied by Terras in \cite{rT1976})
to reconstruct the trajectories of numbers and obtained sufficient conditions in order that two
consecutive numbers have equal total stopping times.  In \cite{gG1993} Gao described computational
results concerning long sequences of integers having identical heights.

Our approach differs from previous ones in its introduction of a recursively defined function
whose values determine which pairs of consecutive even and odd integers have trajectories that
coincide after a specific pair-dependent number of steps.  We will use our main theorem concerning
trajectories (Theorem~\ref{T:t2}) to derive a corresponding result about total stopping
times (Corollary~\ref{C:c2}), and this will imply a number of standard total stopping
time equalities, including the ones in \cite{lG1985}, as well as several novel results.

A simple example may serve to illustrate the power of this approach.

\begin{example}\label{E:e1}
Consider the pair of integers $m = 15$ and $m - 1 = 14$.  A simple calculation reveals that
$T^6(15) = 20 = T^6(14)$ and that $\sigma_{\infty}(15) = 12 = \sigma_{\infty}(14)$.
Writing $m = 2^p \cdot (2q + 1) - 1$, we find that $p = 4$ and $q = 0$, and we observe
that $T^{(p+2)}(m) = T^{(p+2)}(m - 1)$.  If we also write $m = 2n - 1$, then we find that $n = 8$,
and if we look ahead to the next section, we calculate $C(8) = C(2) = 1 - C(1) = 1$.
More generally, we may calculate that $m = 2^{2r} - 1$ yields $T^{(2r+2)}(m) = T^{(2r+2)}(m - 1)$
and $C(2^{(2r-1)}) = 1$.  The trajectories of $m$ and $m - 1$ coincide after the pair-dependent
value of $2r + 2$ steps.

On the other hand, if we consider $m = 31$ and $m - 1 = 30$, then we find $n = 16$, $C(16) = 0$,
and $p = 5$.  But $T^7(31) = 182 \neq 20 = T^7(30)$, and indeed
$\sigma_{\infty}(31) = 67 \neq 13 = \sigma_{\infty}(30)$.  So it appears that the value of
$C(n)$ is determining whether or not the trajectories of pairs of consecutive even and odd
integers coincide after a certain number of steps.  This is the content of Theorem~\ref{T:t2}.
\end{example}

\section{Definition and Properties of $C(n)$}\label{S:defs}
We begin by defining $C: \Z^+ \rightarrow \{0, 1\}$ as follows:
\[
   C(n) =
   \begin{cases}
   0 &\text{if $n = 1$;}\\
   1 - C(n - 2) &\text{if $n$ is odd;}\\
   1 - C(n / 2) &\text{if $n$ is even.}
   \end{cases}
\]
Then by induction $C$ is a well-defined function, and we may calculate the first few
values of $C(n)$ as follows: $C(2) = 1 - C(1) = 1$, $C(3) = 1 - C(1) = 1$, and
$C(4) = 1 - C(2) = 0$.  Additional values of $C(n)$ are readily found by using two
simple propositions.

\begin{proposition}\label{P:p1}
Consider the previously defined function $C(n)$.
\begin{enumerate}
\item $C(2n) = 1 - C(n)$, $C(4n) = C(n)$, $C(2n + 1) = 1 - C(2n - 1)$ , and $C(2n + 3) = C(2n - 1)$ for all $n \ge 1$.
\item Suppose that $n$ is odd, and write $n = 4t + u$, where $t \ge 0$ and $u = 1$ or $u = 3$.  Then $C(n) = C(u)$.
\end{enumerate}
\end{proposition}

\begin{proof}
(1) follows directly from the definition of $C(n)$, while (2) is an immediate consequence of (1).
\end{proof}

Given $n \ge 1$, let $j$ be the largest non-negative integer such that $2^j | n$.  Then
$2^j$ divides $n$, but $2^{j + 1}$ does not divide $n$.  For this $j \ge 0$ we write $2^j
|| n$ and say that $2^j$ \emph{exactly divides} $n$, and we have $n = 2^j \cdot k$,
where $k \ge 1$ is odd.

\begin{proposition}\label{P:p2}
Suppose that $n \ge 1$ and $2^j || n$, and write $n = 2^j \cdot k$, where $j \ge 0$ and $k
\ge 1$ is odd.
\begin{enumerate}
\item If $j$ is even, then $C(n) = C(k)$.
\item If $j$ is odd, then $C(n) = C(2k) = 1 - C(k)$.
\end{enumerate}
\end{proposition}

\begin{proof}
Both of these results follow at once from Proposition~\ref{P:p1}(1).
\end{proof}

Now the following proposition describes precisely when $C(n)$ assumes the values $0$ and $1$.

\begin{proposition}\label{P:p3}
Suppose that $n \ge 1$ and $2^j || n$, and write $n = 2^j \cdot k$, where $j \ge 0$ and $k
\ge 1$ is odd.  Further, write $j = 2r + s$, where $r \ge 0$ and $s = 0$ or $s = 1$, and
write $k = 4t + u$, where $t \ge 0$ and $u = 1$ or $u = 3$.
\begin{enumerate}
\item C(n) = 1 if and only if s = 0 and u = 3 or s = 1 and u = 1.
\item C(n) = 0 if and only if s = 0 and u = 1 or s = 1 and u = 3.
\end{enumerate}
\end{proposition}

\begin{proof}
We consider the value of $C(n)$ in four mutually exclusive and exhaustive cases:
(1) $s = 0$ and $u = 1$; (2) $s = 0$ and $u = 3$; (3) $s = 1$ and $u = 1$; and (4) $s = 1$
and $u = 3$.

If $s = 0$, then $n = 4^r \cdot (4t + u)$, and so $C(n) = C(4t + u) = C(u)$ by Proposition~
\ref{P:p2}(1) and Proposition~\ref{P:p1}(2).  Thus $C(n) = 0$ when $u = 1$, and $C(n) = 1$
when $u = 3$, which handles cases (1) and (2).

If $s = 1$, then $n = 4^r \cdot 2 \cdot (4t + u)$, and so $C(n) = C(2 \cdot (4t + u)) = 1 -
C(4t + u) = 1 - C(u)$ by Proposition~\ref{P:p2}(2) and Proposition~\ref{P:p1}(2).  So
$C(n) = 1$ when $u = 1$, and $C(n) = 0$ when $u = 3$, which takes care of cases (3) and (4)
and completes the proof.
\end{proof}

Our next task will be to establish a connection between the functions $C$ and $T$.  Once
this has been accomplished, we shall be in a position to prove results about clusters of
of integers with equal total stopping times.

\section{Connecting C and T}\label{S:connect}
To establish that connection we need to be able to find the iterates of integers under $T$.
When the integers are written in the proper form, this becomes a straightforward
calculation.  In this section and the following one, we assume that $m$ is an odd positive
integer with $m = 2n - 1$, where $n \ge 2$, and we write $m = 2^p \cdot (2q + 1) - 1$,
where $2^p || (m + 1)$, $p \ge 1$, and $q \ge 0$.

\begin{lemma}\label{L:l1}
The iterates of $m$ and $m - 1$ under $T$ have the following properties:
\begin{enumerate}
\item For $0 \le i \le p$ we have $T^i(m) = 3^i \cdot 2^{p-i} \cdot (2q + 1) - 1$.  In
particular, $T^p(m) = 3^p \cdot (2q + 1) - 1$.
\item $T(m - 1) = 2^{p-1} \cdot (2q + 1) - 1$, and for $1 \le i \le p$ we have $T^i(m - 1)
= 3^{i-1} \cdot 2^{p-i} \cdot (2q + 1) - 1$.  In particular, $T^p(m - 1) = 3^{p-1} \cdot
(2q + 1) - 1$.
\item $T^p(m) = 3 \cdot T^p(m - 1) + 2$.
\item $T^p(m - 1) \equiv 2$ (mod $4$) if and only if $T^p(m) \equiv 0$ (mod $4$).
\item $T^p(m - 1) \equiv 0$ (mod $4$) if and only if $T^p(m) \equiv 2$ (mod $4$).
\end{enumerate}
\end{lemma}
\begin{proof}
We first observe that for any $x \ge 1$ we have $T(2x - 1) = 3x - 1$.
\begin{enumerate}
\item Applying this observation $i$ times to the given $m$ yields the desired result,
and setting $i = p$ then gives the particular formula for $T^p(m)$.
\item Since $m - 1 = 2 \cdot [2^{p-1} \cdot (2q + 1) - 1]$, we see that $T(m - 1) = 2^{p-1}
\cdot (2q + 1) - 1$.  Applying our observation to this expression $i - 1$ times gives the
formula for $T^i(m - 1)$, and when $i = p$, we obtain the stated conclusion.
\item It now follows that $T^p(m) = 3^p \cdot (2q + 1) - 1 = 3 \cdot [3^{p-1} \cdot (2q +
1) - 1] + 2 = 3 \cdot T^p(m - 1) + 2$.
\item From parts (1) and (2) we see that $T^p(m)$ and $T^p(m - 1)$ are both even, and so
each must be congruent to either 0 or 2 modulo 4.  From part (3) we see that $T^p(m) \equiv
2 - T^p(m - 1)$ (mod $4$).  Now both (4) and (5) follow from this.
\end{enumerate}
\end{proof}

Our next result shows that the function $C$ determines whether $T^p(m)$ is congruent to 0
or 2 modulo 4.  This will turn out to be the key to finding clusters of integers with equal
total stopping times.

\begin{theorem}\label{T:t1}
The following statements are equivalent:
\begin{enumerate}
\item $T^p(m) \equiv 0$ (mod $4$);
\item p and q have the same parity (i.e., they are both odd or both even);
\item $C(n) = 1$.
\end{enumerate}
\end{theorem}

\begin{proof}
First note that $n = 2^j \cdot k$, where $j = p - 1 \ge 0$, $2^j || n$, and $k = 2q + 1$.
We also write $j = 2r + s$, where $r \ge 0$ and $s = 0$ or $s = 1$, and $k = 4t + u$, where
$t \ge 0$ and $u = 1$ or $u = 3$.

(2) $\Rightarrow$ (1) and (2) $\Rightarrow$ (3)
Suppose first that $p = 2x$ and $q = 2y$ are both even.  Then $j = 2x - 1$ is odd, so that
$s = 1$, and $k = 4y + 1$, so that $u = 1$.  Thus $C(n) = 1$ by Proposition~\ref{P:p3}(1),
and (3) holds.  Now $3^p \equiv 1$ (mod $4$) and $2q + 1 = 1$ (mod $4$), and so by Lemma~
\ref{L:l1}(1) $T^p(m) = 3^p \cdot (2q + 1) - 1 \equiv 0$ (mod $4$), which means that (1) holds.

Next suppose that $p = 2x + 1$ and $q = 2y + 1$ are both odd.  Then $j = 2x$ is even, which
implies that $s = 0$, and $k = 4y + 3$, which makes $u = 3$.  So $C(n) = 1$ by Proposition~
\ref{P:p3}(1), and once again (3) holds.  Now $3^p \equiv -1$ (mod $4$) and $2q + 1 \equiv
-1$ (mod $4$), and so $T^p(m) = 3^p \cdot (2q + 1) - 1 \equiv 0$ (mod $4$), which shows that
(1) holds.

(1) $\Rightarrow$ (2) and (3) $\Rightarrow$ (2)
Assume now that $p$ and $q$ do not have the same parity.  If $p = 2x$ is even and $q = 2y
+ 1$ is odd, then $3^p \equiv 1$ (mod $4$) and $2q + 1 \equiv -1$ (mod $4$), and so $T^p(m)
= 3^p \cdot (2q + 1) - 1 \equiv 2$ (mod $4$).  Thus (1) does not hold.  Since $j = 2x - 1$
is odd, we have $s = 1$, and since $k = 4y + 3$, we have $u = 3$.  So $C(n) = 0$ by
Proposition~\ref{P:p3}(2), and (3) does not hold in this case.

Finally, suppose that $p = 2x + 1$ is odd and $q = 2y$ is even.  Then $3^p \equiv -1$ (mod
$4$) and $2q + 1 \equiv 1$ (mod $4$), and again $T^p(m) = 3^p \cdot (2q + 1) - 1 \equiv 2$
(mod $4$).  So (1) does not hold.  Since $j = 2x$ is even, we see that $s = 0$, and since
$k = 4y + 1$, we have $u = 1$.  Once again $C(n) = 0$ by Proposition~\ref{P:p3}(2), and (3)
does not hold.  This completes the proof.
\end{proof}

As a useful corollary we may restate this as follows.

\begin{corollary}\label{C:c1}
The following statements are equivalent:
\begin{enumerate}
\item $T^p(m) \equiv 2$ (mod $4$);
\item $p$ and $q$ have opposite parities;
\item $C(n) = 0$.
\end{enumerate}
\end{corollary}

\section{Clusters of Consecutive Integers}\label{S:clusters}
We turn now to our main result, which will be the basis for the determination of clusters.
We defined the relationships assumed to hold for $m$, $n$, and $p$ at the beginning of the
previous section.  Note that $p + 2$ is the pair-dependent number of steps at which
the trajectories of $m - 1$ and $m$ first coincide.

\begin{theorem}\label{T:t2}
$T^{p+2}(m - 1) = T^{p+2}(m)$ if and only if $C(n) = 1$.
\end{theorem}

\begin{proof}
First suppose that $C(n) = 1$.  Then $T^p(m) \equiv 0$ (mod $4$) by Theorem~
\ref{T:t1}, and so by Lemma~\ref{L:l1}(4) we have $T^p(m - 1) \equiv 2$ (mod $4$).  It
follows that both $T^p(m)$ and $T^p(m - 1)$ are even, while $T^{p+1}(m) = T^p(m) / 2$ is
even and $T^{p+1}(m - 1) = T^p(m - 1) / 2$ is odd.  Hence $T^{p+2}(m) = T^p(m) / 4$, and
using Lemma~\ref{L:l1}(3) we find
\[
\begin{split}
T^{p+2}(m - 1) &= T(T^{p+1}(m - 1)) = T(T^p(m - 1) / 2) = (3 \cdot (T^p(m - 1) /2 ) + 1) / 2 \\
&= (3 \cdot (T^p(m - 1)) + 2) / 4 = T^p(m) / 4 \\
&= T^{p+2}(m).
\end{split}
\]

Now suppose that $C(n) = 0$.  Then by Corollary~\ref{C:c1} $T^p(m) \equiv 2$ (mod $4$), and
so by Lemma~\ref{L:l1}(5) $T^p(m - 1) \equiv 0$ (mod $4$).  So $T^{p+2}(m - 1) =
T^p(m - 1) / 4$, and we calculate
\[
\begin{split}
T^{p+2}(m) &= T(T^{p+1}(m)) = T(T^p(m) / 2) = (3 \cdot (T^p(m)/2) + 1) / 2 \\
&= (3 / 4)\cdot T^p(m) + 1 / 2 = (3 / 4) \cdot (3 \cdot T^p(m - 1) + 2) + 1 / 2 \\
&= (9 / 4) \cdot T^p(m - 1) + 2.
\end{split}
\]

Suppose now, to reach a contradiction, that $T^{p+2}(m - 1) = T^{p+2}(m)$.  Then
$T^p(m - 1) / 4 = (9 / 4) \cdot T^p(m - 1) + 2$, and this implies that $T^p(m - 1) = -1$, which
is impossible.  Hence we conclude that $T^{p+2}(m - 1) \neq T^{p+2}(m)$, and the proof is
complete.
\end{proof}

As a corollary we will obtain our first result about clusters.  The following lemma
explains why the restriction $n \ge 4$ is necessary in that corollary.

\begin{lemma}\label{L:l2}
Suppose that $C(n) = 1$.  Then $T^i(m - 1) = 1$ for some $i$ with $1 \le i \le p + 1$ if
and only if $n = 2$ or $n = 3$.
\end{lemma}

\begin{proof}
If $n = 2$, then $m = 3$ and $p = 2$, and so $T(m - 1) = T(2) = 1$.  If $n = 3$, then
$m = 5$ and $p = 1$, and we have $T^2(m - 1) = T^2(4) = 1$.

For the converse, assume that $T^i(m - 1) = 1$ with $1 \le i \le p + 1$.  If $1 \le i \le
p$, then by Lemma~\ref{L:l1}(2) we have $3^{i-1} \cdot 2^{p-i} \cdot (2q + 1) - 1 =
T^i(m - 1) = 1$, and so $3^{i-1} \cdot 2^{p-i} \cdot (2q + 1) = 2$.  It follows that
$3^{i-1} = 1$, $2^{p-i} = 2$, and $2q + 1 = 1$, so that $i = 1$, $p = 2$, and $q = 0$.
This gives $m = 3$ and $n = 2$.  If $i = p + 1$, then $(3^{p-1} \cdot (2q + 1) - 1) / 2 =
T^{p+1}(m - 1) = 1$, and so $3^{p-1} \cdot (2q + 1) = 3$.  There are two cases to consider:
(1) $p = 1$ and $q = 1$ and (2) $p = 2$ and $q = 0$.  In the first case we have $m = 5$ and
$n = 3$, while in the second case we have $m = 3$ and $n = 2$ once again.  (This happens
because $T^3(2) = T(2) = 1$.)  This completes the proof.
\end{proof}

\begin{corollary}\label{C:c2}
If $n \ge 4$ and $C(n) = 1$, then $m$ and $m - 1$ form a cluster, or
equivalently, $\sigma_{\infty}(2n - 2) = \sigma_{\infty}(2n - 1)$.
\end{corollary}

\begin{proof}
Let $m_{1}$ be the common value of $T^{p+2}(m - 1)$ and $T^{p+2}(m)$ according to Theorem~\ref{T:t2}.
If $\sigma_{\infty} (m_{1})$ is infinity, then $\sigma_{\infty}(m - 1)$ and $\sigma_{\infty}(m)$
must both be infinity as well.  So suppose that $\sigma_{\infty}(m_{1})$ is finite.  Since $n \ge 4$,
$m - 1 \ge 6$, and by Lemma~\ref{L:l2} $T^i(m - 1) > 1$ for all $i$ with $1 \le i \le p + 1$.
So $m_{1} = T^{p+2}(m - 1) \ge 1$, and it follows that $\sigma_{\infty}(m - 1) =
\sigma_{\infty}(m_{1}) + p + 2$.  By Lemma~\ref{L:l1}(1) $m$ and its iterates $T^i(m)$ form
a strictly increasing sequence, and so $T^i(m) > 1$ for $0 \le i \le p$.  Since $4 |
T^p(m)$ by Theorem~\ref{T:t1}, $T^{p+1}(m) > 1$ and $T^{p+2}(m) \ge 1$.  Thus
$\sigma_{\infty}(m) = \sigma_{\infty}(m_{1}) + p + 2 = \sigma_{\infty}(m - 1)$ so that $m$
and $m - 1$ form a cluster.
\end{proof}

\begin{remark}\label{R:r1}
The converse of Corollary~\ref{C:c2} is false.  Consecutive even and odd integers
may form a cluster when $C(n) = 0$.  This happens frequently with larger values of $n$.
The smallest value of $n$ for which this happens is $n = 121$: $\sigma_{\infty}(240) =
\sigma_{\infty}(241) = 16$, but $C(121) = C(1) = 0$.

In this case $p = 1$, but $T^3(240) = 30 \neq 272 = T^3(241)$, and so the trajectories
fail to coincide after $p + 2$ steps.  Indeed, $T^i(240) \neq T^i(241)$ for
$1 \leq i \leq 9$, and $T^{10}(240) = 20 =  T^{10}(241)$.
\end{remark}

\section{Further Consequences}\label{S:conseq}
We now show how Corollary~\ref{C:c2} implies two standard results about total stopping
times, and we then use it to derive several new ones.

\begin{corollary}\label{C:c3}
$\sigma_{\infty}(8i + 4) = \sigma_{\infty}(8i + 5)$ for all $i \ge 1$.
\end{corollary}

\begin{proof}
Let $n = 4i + 3$, where $i \ge 1$.  Then by Proposition~\ref{P:p1}(2) $C(n) = C(3) = 1$.
Since $2n - 2 = 8i + 4$, $\sigma_{\infty}(8i + 4) = \sigma_{\infty}(8i + 5)$ by Corollary~
\ref{C:c2}.
\end{proof}

\begin{corollary}\label{C:c4}
$\sigma_{\infty}(16i + 2) = \sigma_{\infty}(16i + 3)$ for all $i \ge 1$.
\end{corollary}

\begin{proof}
Put $n = 8i + 2$, where $i \ge 1$.  Then $C(n) = 1 - C(4i + 1) = 1 - C(1) = 1$.  Since
$2n - 2 = 16i + 2$, Corollary~\ref{C:c2} tells us that $\sigma_{\infty}(16i + 2) =
\sigma_{\infty}(16i + 3)$.
\end{proof}

\begin{remark}\label{R:r2}
Theorem 1 of Garner \cite{lG1985} is another consequence of Corollary~\ref{C:c2}.  To see this,
we observe that the value of the even integer $n$ in Garner's Theorem 1 corresponds to a value of
$2n_1(i) - 2$ in our Corollary~\ref{C:c2}.  Therefore
\[
n_1(i) = 2^{(i+2)} m + 2^{(i+1)} + (-1)^i 2^i = 2^i (4m + 2 + (-1)^i).
\]
Since $C(n_1(i)) = 1$ when $i = 0$ and $i = 1$ and $n_1(i + 2) = 4n_1(i)$, it follows that $C(n_1(i)) = 1$
for all $i \ge 0$ and all $m \ge 0$.  Thus our Corollary~\ref{C:c2} implies Garner's Theorem 1.
\end{remark}

We turn next to some other novel consequences of Corollary~\ref{C:c2}.

\begin{corollary}\label{C:c5}
If $i \ge 1$ and $C(i) = 1$, then $\sigma_{\infty}(8i - 2) = \sigma_{\infty}(8i - 1)$.
\end{corollary}

\begin{proof}
Set $n = 4i$, where $i \ge 1$ and $C(i) = 1$.  Then $C(n) = C(i) = 1$, and $2n - 2 =
8i - 2$ with $n \ge 4$.  Hence by Corollary~\ref{C:c2} we have $\sigma_{\infty}(8i - 2) =
\sigma_{\infty}(8i - 1)$.
\end{proof}

\begin{corollary}\label{C:c6}
If $i \ge 2$ and $C(i) = 0$, then $\sigma_{\infty}(4i - 2) = \sigma_{\infty}(4i - 1)$.
\end{corollary}

\begin{proof}
Let $n = 2i$, where $i \ge 2$ and $C(i) = 0$.  Now $2n - 2 = 4i - 2$ with $n \ge 4$, and
$C(n) = 1 - C(i) = 1$.  Once again Corollary~\ref{C:c2} yields the desired conclusion.
\end{proof}

\begin{corollary}\label{C:c7}
If $i \ge 2$ and $C(i) = 0$, then $\sigma_{\infty}(8i - 4) = \sigma_{\infty}(8i - 3) =
\sigma_{\infty}(8i - 2)$.
\end{corollary}

\begin{proof}
By Corollary~\ref{C:c6} we have $\sigma_{\infty}(4i - 2) = \sigma_{\infty}(4i - 1)$, and so
$\sigma_{\infty}(8i - 4) = \sigma_{\infty}(8i - 2)$.  Now from Corollary~\ref{C:c3} we
obtain $\sigma_{\infty}(8i - 4) = \sigma_{\infty}(8i - 3)$, and the result is proved.
\end{proof}

\begin{remark}\label{R:r3}
Theorem 2 of Garner is also a consequence of our results.  We can prove this by explicitly
writing out the integers in Tables 2 and 3 (\cite{lG1985}, pp.~60-61) and then applying
our corollaries.  Considering Table 2, we see that the rule for $n$ is given by
$n = 8i(j) - 3$, where
\[
i(j) = 2^{(j+2)} m + 2^j(2 + (-1)^{(j+1)}) = 2^j(4m + 2 + (-1)^{(j+1)})
\]
for all $j \ge 0$ and all $m \ge 0$.  Since $C(i(j)) = 0$ when $j = 0$ and $j = 1$ and
$i(j + 2) = 4i(j)$, it follows that $C(i(j))) = 0$ for all $j \ge 0$ and all $m \ge 0$.
Since $i(j) \ge 2$ whenever $j \ge 0$, $m \ge 0$, and $(j,m) \neq (0,0)$, our Corollary~\ref{C:c7}
implies the first half of Garner's Theorem 2.  (We excluded the case in which $(j,m) = (0,0)$
because $\sigma_{\infty}(5) \neq \sigma_{\infty}(6)$.) 
\end{remark}

To complete the proof of Garner's Theorem 2 we require one other corollary.

\begin{corollary}\label{C:c8}
If $i \ge 1$ and $C(9i + 6) = 1$, then $\sigma_{\infty}(32i + 17) = \sigma_{\infty}(32i + 18)$.
\end{corollary}

\begin{proof}
Let $m = 32i + 17$.  Then we calculate $T(m) = 48i + 26$, $T^2(m) = 24i + 13$, $T^3(m) = 36i + 20$,
and $T^4(m) = 18i + 10$, while $T(m + 1) = 16i + 9$, $T^2(m + 1) = 24i + 14$, $T^3(m + 1) = 12i + 7$,
and $T^4(m + 1) = 18i + 11$.  Since $18i + 10 = 2n - 2$, where $n = 9i + 6 \ge 15$, and by hypothesis
$C(n) = C(9i + 6) = 1$, Corollary~\ref{C:c2} implies that
$\sigma_{\infty}(18i + 10) = \sigma_{\infty}(18i + 11)$.  Hence
$\sigma_{\infty}(m) = \sigma_{\infty}(18i + 10) + 4 = \sigma_{\infty}(18i + 11) + 4 = \sigma_{\infty}(m + 1)$.
\end{proof}

\begin{remark}\label{R:r4}
Returning now to Table 3 of Garner's paper \cite{lG1985}, we see that the rule for $n$ is given by
$n = 32i(j) + 17$, where
\[
i(j) = 2^{(j-1)}(8m + (3 + (-1)^{(j+1)})) + (1/3) \cdot (2^{(j-1)}(3 + (-1)^j) - 2)
\]
for all $j \ge 0$ and all $m \ge 0$.  Therefore
\[
9i(j) + 6 = 2^{(j-1)}(72m + 36 + (-1)^{(j+1)} \cdot 6).
\]
Now when $j = 0$, we have $9i(j) + 6 = 36m + 15$ so that $C(9i(j) + 6) = C(36m + 15) = 1$,
and when $j = 1$, we have $9i(j) + 6 = 72m + 42$ so that $C(9i(j) + 6) = 1 - C(36m + 21) = 1$.
Since $9i(j + 2) + 6 = 4 \cdot (9i(j) + 6)$, it follows that $C(9i(j) + 6) = 1$
for all $j \ge 0$ and all $m \ge 0$.  So our Corollary~\ref{C:c8} implies the second half of
Garner's Theorem 2.
\end{remark}

\section{Acknowledgements}\label{S:ack}
The author would like to dedicate this paper to the memory of George Cain, Professor
Emeritus in the School of Mathematics at the Georgia Institute of Technology, who always
had ``all the time in the world" for his students.

The author would also like to thank Jeffrey Lagarias for generously supplying helpful
references to previous work on total stopping times and heights in the 3X + 1 Problem.

\medskip

\noindent MSC2010: 11B37, 11B83, 11B50

\end{document}